\documentclass[12pt]{article}

\usepackage{amsthm,amsmath,amsfonts,amssymb,graphicx,ascmac,color,bm}
\usepackage{natbib}

\theoremstyle{definition}
\newtheorem{theorem}{Theorem}

\newtheorem{lemma}{Lemma}
\newtheorem{definition}{Definition}

\newtheorem{example}{Example}
\newtheorem{remark}{Remark}
\newtheorem{assumption}{Assumption}

\def\T{{ \mathrm{\scriptscriptstyle T} }}

\begin{document}

\title{Infinitely imbalanced binomial regression\\
and deformed exponential families}
\author{T. Sei}

\maketitle

\begin{abstract}
The logistic regression model is known to
converge to a Poisson point process model 
if the binary response tends to infinitely imbalanced.
In this paper, it is shown that this phenomenon is universal
in a wide class of link functions on binomial regression.
The proof relies on the extreme value theory.
For the logit, probit and complementary log-log link functions,
the intensity measure of the point process becomes an exponential family.
For some other link functions, deformed exponential families appear.
A penalized maximum likelihood estimator for the Poisson point process model is suggested.
\end{abstract}

\section{Introduction}

Let $\{(X_i,Y_i)\}_{i=1}^m$
be $m$ independently and identically distributed observable data on $\mathbb{R}^p\times\{0,1\}$.
The conditional distribution of $Y_i$ given $X_i$ is assumed to be
\begin{align}
 P(Y_i=1\mid X_i,a,b) = G(a + b^\T X_i),
 \quad a\in\mathbb{R},
 \quad b\in\mathbb{R}^p,
 \label{eq:model}
\end{align}
where $G(\cdot)$ is a one-dimensional cumulative distribution function.
The inverse function $G^{-1}(p)=\sup\{z: G(z)\leq p\}$ is the link function
in terms of generalized linear models.
Denote the marginal distribution of $X_i$ by $F(dX_i)$.
The distribution function $G$ is typically the logistic, standard normal or Gumbel distributions.
The corresponding link functions are the logit, probit and complementary log-log functions, respectively.
For the three examples, the log-likelihood function of (\ref{eq:model}) is concave; see \citet{Wedderburn1976}.

Our interest is the situation that the data is highly imbalanced.
In other words, the probability of success is almost zero.
Examples of such cases are fraud detection, medical diagnosis, political analysis and so forth.
See e.g.\ \citet{Bolton2002}, \citet{Chawla2004}, \citet{Jin2005}, and \citet{King2001}.
For the data without covariates,
Poisson's law of rare events is well known:
if $P(Y_i=1)=\lambda/m + o(m^{-1})$, then
the probability distribution of $\sum_{i=1}^m Y_i$ converges to the Poisson distribution
with the mean parameter $\lambda$.
From this observation, for highly imbalanced data,
it is natural to consider that the true parameter $(a,b)$ in (\ref{eq:model})
depends on $m$, say $(a_m,b_m)$, and $G(a_m)\to 0$ as $m\to\infty$.

\citet{Owen2007} showed that the maximum likelihood estimator of the logistic regression model converges to
that of an exponential family if $\sum_{i=1}^m Y_i$ is fixed and $m$ goes to infinity.
This result is roughly derived as follows.
Consider the model (\ref{eq:model}) with the logistic distribution $G(z)=e^z/(1+e^z)$.
Take $a_m(\alpha)=-\log m+\alpha$ and $b_m(\beta)=\beta$ for any fixed $\alpha$ and $\beta$.
Then we obtain
\begin{align}
 P(Y_i=1\mid X_i,a_m(\alpha),b_m(\beta))
 = \frac{e^{-\log m+\alpha+\beta^\T X_i}}{1+e^{-\log m+\alpha+\beta^\T X_i}}
 = \frac{e^{\alpha + \beta^\T X_i}}{m} + o(m^{-1})
 \label{eq:Owen}
\end{align}
as $m\to\infty$.
By Bayes' theorem,
the conditional density of $X_i$ given $Y_i=1$ with respect to the distribution $F(dX_i)$ is, at least formally,
\begin{align}
 \frac{e^{\beta^\T X_i}}{\int e^{\beta^\T x}F(dx)} + o(1).
 \label{eq:Owen-conditional}
\end{align}
This is an exponential family with the sufficient statistic $x_i$, and Owen's result follows.

\begin{remark} \label{rem:Owen}
 To be precise, \cite{Owen2007} proved the convergence result under a different setting from here.
 He assumed that the true conditional distribution of $X_i$ given $Y_i=j$, $j\in\{0,1\}$,
 is any distribution $F_j$.
 In our setting, $F_0$ is asymptotically equal to $F$,
 and the density of $F_1$ with respect to $F$ should satisfy
 (\ref{eq:Owen-conditional}).
 In other words, our setting becomes misspecified unless this equality is satisfied.
 We discuss this point again in Section~\ref{section:discussion}.
\end{remark}

\citet{Warton2010} pointed out that
the likelihood of logistic regression converges to a Poisson point process model
with a specific form of intensity.
Indeed, by (\ref{eq:Owen}),
the probability $P(Y_i=1, X_i\in A)$
is approximately $m^{-1}\int_A e^{\alpha+\beta^\T x}F(dx)$ for any compact subset $A$ of $\mathbb{R}^p$.
Therefore, by Poisson's law of rare events,
the number of observations $X_i$ for which $X_i\in A$ and $Y_i=1$
is approximately distributed according to the Poisson distribution
with mean $\int_A e^{\alpha+\beta^\T x}F(dx)$.
This is the Poisson point process with the intensity measure $e^{\alpha+\beta^\T x}F(dx)$.

In this paper, we consider the limit of various binomial regression models
other than the logistic model.
As expected from the result on logistic regression,
the limit becomes a Poisson point process.
A remarkable fact we prove is that the intensity measure of the point process
should be a $q$-exponential family for some real number $q$.
The $q$-exponential family, also called the deformed exponential family or $\alpha$-family,
is recently much investigated in the literature of statistical physics and information geometry;
see e.g.\ \citet{Amari1985}, \citet{AmariNagaoka2000}, \citet{AmariOhara2011}, \citet{Naudts2002}, \citet{Naudts2010}, and \citet{Tsallis1988}.
The precise definition is given in Section~\ref{section:main}.
The proof relies on the theory of extreme values.
For example, for the probit or complementary log-log link functions,
the limit of binomial regression is the usual exponential family as with the logit link.
On the other hand, if $G$ is the Cauchy distribution,
then the limit becomes a $q$-exponential family with $q=2$.
If the uniform distribution is used, $q=0$.

As a related work, \citet{Ding2011}
introduced the $t$-logistic regression,
that uses the $q$-exponential family for binary response, where $q=t$.
In Section~\ref{section:example},
we show that the $t$-logistic regression
converges to the $q$-exponential family if $q\geq 0$.

In Section~\ref{section:estimator}, we study a penalized maximum likelihood estimator
on the $q$-exponential family of intensity measures.
For some special cases, the estimator is reduced to a known admissible estimator
for the Poisson mean parameter; see \citet{GhoshYang1988}.

Some related problems are discussed in Section~\ref{section:discussion}.

\section{Imbalanced asymptotics of binomial regression} \label{section:main}

For each real number $q$, define the $q$-exponential function by
\begin{align}
 \exp_q(z) =
 \left\{\begin{array}{ll}
  e^z, & \mbox{if}\ q=1,\\
  {[1 + (1-q) z]}_+^{1/(1-q)},& \mbox{if}\ q\neq 1,
 \end{array}
 \right.
 \label{eq:expq}
\end{align}
where $[z]_+=\max(z,0)$ and $[0]_+^{-1}=\infty$.
This is inverse of the Box-Cox transformation.
Note that 
$\exp_q(z)=\infty$ for $z\geq -1/(1-q)$ if $q>1$.
The function $\exp_q(z)$ is convex if and only if $q\geq 0$.

Consider the binomial regression model (\ref{eq:model})
and put the following assumption on the distribution function $G$.

\begin{assumption} \label{assumption:GEV}
 There exist $q>0$, $c_m\in\mathbb{R}$ and $d_m>0$ such that
 \begin{align}
  G(c_m + d_m z) = \frac{1}{m}\exp_q(z) + o(m^{-1})
    \label{eq:GEV}
 \end{align}
 as $m\to\infty$ for each $z\in\mathbb{R}$.
\end{assumption}

In the extreme value theory,
it is known that there is no other asymptotic form than (\ref{eq:GEV})
as long as it exists; see e.g.\ \citet[Theorem 1.1.2 and 1.1.3]{deHaan2006}.
The number $q$ controls the lower tail structure of $G$.
For example, the logistic distribution satisfies
Assumption~\ref{assumption:GEV}
with $q=1$, $c_m=-\log m$ and $d_m=1$.
Other examples including the normal and Cauchy distributions are considered in Section~\ref{section:example}.

We define
\begin{align}
a_m(\alpha)=c_m+d_m\alpha
\quad \mbox{and}\quad b_m(\beta)=d_m\beta
\label{eq:a_m-b_m}
\end{align}
for $(\alpha,\beta)\in\mathbb{R}\times\mathbb{R}^p$
by using the sequences $c_m$ and $d_m$ that satisfy (\ref{eq:GEV}).
Denote the probability law of $\{(X_i,Y_i)\}_{i=1}^m$
under the true parameter $(a_m(\alpha),b_m(\beta))$
by $P_{m,\alpha,\beta}$.

Now the asymptotic form like (\ref{eq:Owen}) follows from the assumption.
Indeed,
\begin{align*}
 P_{m,\alpha,\beta}(Y_i=1\mid X_i)
 &= G(a_m(\alpha)+b_m(\beta)^\T X_i) \\
 &= G(c_m+d_m(\alpha+\beta^\T X_i)) \\
 &= \frac{1}{m}\exp_q(\alpha+\beta^\T X_i) + o(m^{-1}).
\end{align*}
Therefore, as in the logistic regression, we expect that the binomial regression model with $G$
converges to the Poisson point process under Assumption~\ref{assumption:GEV}.

We give a lemma before the main result.

\begin{lemma} \label{lem:Owen-mar-cond}
 Let $(\alpha,\beta)\in\mathbb{R}\times\mathbb{R}^p$.
 Let $A$ be any compact subset of $\mathbb{R}^p$
 such that the function $\exp_q(\alpha+\beta^\T x)$ is finite over $x\in A$.
 Then the following equation holds:
 \begin{align}
  P_{m,\alpha,\beta}(Y_i=1, X_i\in A)
  &= \frac{\lambda(A)}{m} + o(m^{-1}),
  \label{eq:lem-marginal}
 \end{align}
 where $\lambda(A)=\int_A\exp_q(\alpha+\beta^\T x)F(dx)$.
\end{lemma}

The proof of Lemma~\ref{lem:Owen-mar-cond} is given in Appendix.

\begin{theorem}
 \label{theo:main1}
 Denote the observations $X_i$ for which $Y_i=1$ by $\{x_i\}_{i=1}^n$.
 Then, under $P_{m,\alpha,\beta}$,
 the set $\{x_i\}_{i=1}^n$
 converges in law to the Poisson point process with
 the intensity measure
 \begin{align}
 \lambda(dx)=\exp_q(\alpha+\beta^\T x)F(dx)
  \label{eq:q-exponential-intensity}
 \end{align}
 as $m\to\infty$.
 More precisely, we have
 \begin{align}
  \lim_{m\to\infty} P_{m,\alpha,\beta}
  \left(\#\{i\mid x_i\in A_j\} = n_j,\ j=1,\ldots,J\right)
  = \prod_{j=1}^J \frac{\lambda(A_j)^{n_j}e^{-\lambda(A_j)}}{n_j!}
   \label{eq:main1}
 \end{align}
 for any positive integer $J$, non-negative integers $n_j$
 and mutually disjoint compact subsets $A_j$ of $\mathbb{R}^p$
 such that $\exp_q(\alpha+\beta^\T x)$ is finite over $x\in A_j$.
\end{theorem}

The equation (\ref{eq:main1}) is consistent with
the definition of weak convergence of point processes; see \citet{Embrechts1997}.

\begin{proof}[Proof of Theorem~\ref{theo:main1}]
Define
\begin{align*}
 x(A) &= \#\{i\in\{1,\ldots,n\}\mid x_i\in A\}\\
 &= \#\{i\in\{1,\ldots,m\}\mid (X_i,Y_i)\in A\times \{1\}\}.
\end{align*}
Since $\{(X_i,Y_i)\}_{i=1}^m$ is an independent and identically distributed sequence,
the random vector $(x(A_1),\ldots,x(A_J))$ for the disjoint compact subsets $\{A_j\}_{j=1}^J$
is distributed as the multinomial distribution.
Then, by Lemma~\ref{lem:Owen-mar-cond} and Poisson's law of rare events,
$(x(A_1),\ldots,x(A_J))$ converges to independent Poisson random variables
with intensity $(\lambda(A_1),\ldots,\lambda(A_J))$.
The proof is completed.
\end{proof}

By Theorem~\ref{theo:main1}, the logistic regression model converges to the Poisson point process model
with intensity $\exp(\alpha+\beta^\T x)F(dx)$ as \citet{Warton2010} showed.

\begin{definition}
 For each $q\in\mathbb{R}$,
 we call the set of intensity measures (\ref{eq:q-exponential-intensity})
  the $q$-exponential family of intensity measures.
 Denote the law of the process $\{x_i\}_{i=1}^n$ with respect to (\ref{eq:q-exponential-intensity}) by $P_{\alpha,\beta}^{(q)}$.
\end{definition}

The $q$-exponential family of intensity measures is closely related to
the $q$-exponential family of probability measures as follows.
Denote the total intensity by
\begin{align}
\Lambda_q(\alpha,\beta)
= \int_{\mathbb{R}^p} \exp_q(\alpha+\beta^\T x)F(dx).
\label{eq:total-intensity}
\end{align}
Assume $\Lambda_q(\alpha,\beta)<\infty$.
Then the likelihood of $P_{\alpha,\beta}^{(q)}$ is
\begin{align}
 \frac{e^{-\Lambda_q(\alpha,\beta)}}{n!} \prod_{i=1}^n \exp_q(\alpha+\beta^\T x_i),
  \label{eq:limit-likelihood}
\end{align}
where
the base measure of $n$ is the counting measure on $\{0,1,\cdots\}$,
and the base measure of $x_i$ for each $i$ is the distribution $F(dx_i)$.
In (\ref{eq:limit-likelihood}),
the number $n$ of observed points
is marginally distributed according to the Poisson distribution with intensity $\Lambda_q(\alpha,\beta)$.
Each point $x_i$ is independently distributed according to
 the $q$-exponential family defined by the probability density function
 \begin{align}
  \frac{\exp_q(\alpha+\beta^\T x_i)}{\Lambda_q(\alpha,\beta)}
   \label{eq:q-exponential-family}
 \end{align}
 with respect to $F(dx)$.
The $q$-exponential family is also called
 the deformed exponential family or the
 $\alpha$-family; see \citet{AmariNagaoka2000} for the $\alpha$-family,
where $\alpha=2q-1$ should be distinguished with the regression coefficient $\alpha$.
It is known that the density (\ref{eq:q-exponential-family}) is also written as
$\exp_q(\theta^\T x_i-\psi_q(\theta))$
with appropriate $\theta$ and $\psi_q(\theta)$; see  e.g.\ \citet{AmariOhara2011}.
However, we do not use this parametrization since
the quantity $\Lambda_q(\alpha,\beta)$ remains in
the whole likelihood (\ref{eq:limit-likelihood}).


We conjecture that the maximum likelihood estimator of the binomial regression model $P_{m,\alpha,\beta}$ converges to that of
the Poisson process model $P_{\alpha,\beta}^{(q)}$ under mild conditions.
However, we only give experimental results in Section~\ref{section:example}.
Instead, we study the estimation problem of the limit model $P_{\alpha,\beta}^{(q)}$ in Section~\ref{section:estimator}.
See also Section~\ref{section:discussion} for further discussion.

\section{Examples} \label{section:example}

In this section, we give some examples of distributions $G$
satisfying Assumption~\ref{assumption:GEV},
and experimental results on the maximum likelihood estimation.

 Even if $G$ satisfies Assumption~\ref{assumption:GEV},
 the sequences $c_m$ and $d_m$ are not uniquely determined.
 A unified choice is known  (see \citet[Theorem~2.1.4--2.1.6]{Galambos1987}).
 However, in the following examples, one of possible pairs $(c_m,d_m)$
 is explicitly given for each case.

 For the logistic distribution and the Gumbel distribution $G(z)=1-\exp(-e^z)$ on minimum values, we have
 \begin{align}
 q=1,
 \quad c_m = -\log m,
 \quad d_m=1.
  \label{eq:logistic-c_m}
 \end{align}
 For the standard normal distribution, we have
  \begin{align}
 q=1,
 \quad c_m=-(2\log m)^{1/2} + \frac{\log(\log m) + \log(4\pi)}{2(2\log m)^{1/2}},
 \quad d_m=(2\log m)^{-1/2}.
 \label{eq:normal-c_m}
 \end{align}
 See e.g.\ \citet[Section 2.3.2]{Galambos1987}.
 For the Cauchy distribution, we have
 \begin{align}
  q=2,
  \quad c_m = -m/\pi,
  \quad d_m=m/\pi.
  \label{eq:Cauchy-c_m}
 \end{align}
 For other examples such as $t$-distribution and Pareto distributions,
 refer to \citet{Galambos1987} and \citet{Embrechts1997}.

 We briefly study the $t$-logistic regression proposed by \citet{Ding2011}.
 For each real number $t$, let $G_t(z)=\exp_t(z-\gamma_t(z))$,
 where $\exp_t$ denotes the $q$-exponential function with $q=t$
 and $\gamma_t(z)$ is uniquely determined by
 \begin{align}
 \exp_t(z-\gamma_t(z)) + \exp_t(-\gamma_t(z)) = 1.
  \label{eq:t-logistic-cond}
 \end{align}
 We call $G_t(z)$ the $t$-logistic distribution.
 Uniqueness of $\gamma_t(z)$ follows from
 strictly monotone property of the $q$-exponential function.
 The distribution $G_t(z)$ is symmetric in the sense that $G_t(-z)=1-G_t(z)$
 since $\gamma_t(-z)=-z+\gamma_t(z)$ by (\ref{eq:t-logistic-cond}).
 We obtain the following theorem.
 The proof is given in Appendix.

\begin{theorem} \label{theo:t-logistic}
 The $t$-logistic distribution $G_t$ satisfies Assumption~\ref{assumption:GEV}
 with $q=\max(t,0)$.
\end{theorem}

Table~\ref{table:jikken1} and Table~\ref{table:jikken2} show the experimental results.
The sample is
\begin{align}
 (X_i, Y_i)
 = 
 \left\{
  \begin{array}{l}
  \left(0.4 + 0.4(i-1)/(n-1),1\right)
  \quad \mbox{if} \quad i\in\{1,\ldots,n\},
  \\
  \left((i-n-1)/(m-n-1),0\right)
  \quad \mbox{if}\quad i\in\{n+1,\ldots,m\}
  \end{array}
  \right.
  \label{eq:sample}
\end{align}
for $n=10$ and various $m$'s.
For the binomial regression models, the estimated regression coefficient $(\hat{a},\hat{b})$ is normalized by (\ref{eq:a_m-b_m}).
From Table~\ref{table:jikken1},
the convergence rate for the probit link is very slow, or may not converge.
For the others, the rate is satisfactory.

\begin{table}[ht]
\caption{Comparison of the maximum likelihood estimate of the Poisson point process model with $q=1$ and
the binomial regression models.
The logit, probit and clogog (complementary log-log) link functions are used.
The sample is (\ref{eq:sample}) and $n$ is fixed to $10$.
The normalizing sequence $(c_m,d_m)$ is (\ref{eq:logistic-c_m}) and (\ref{eq:normal-c_m}).
}
\label{table:jikken1}
\begin{center}
\begin{tabular}{|c|cc|cc|cc|cc|}
  \hline
 & \multicolumn{2}{|c|}{Poisson process} & \multicolumn{2}{|c|}{logit}
 & \multicolumn{2}{|c|}{probit} & \multicolumn{2}{|c|}{cloglog} \\
 $m$& $\hat{\alpha}$ & $\hat{\beta}$& $\hat{\alpha}$ & $\hat{\beta}$& $\hat{\alpha}$ & $\hat{\beta}$ & $\hat{\alpha}$ & $\hat{\beta}$\\
  \hline
 $10^2$ & 1.6504 & 1.1737 & 1.6883 & 1.3067 & 2.0282 & 2.3030 & 1.6975 & 1.1883 \\ 
 $10^3$ & 1.6277 & 1.2246 & 1.6314 & 1.2373 & 1.9070 & 1.8777 & 1.6322 & 1.2260 \\ 
 $10^4$ & 1.6256 & 1.2294 & 1.6260 & 1.2307 & 1.8634 & 1.6725 & 1.6260 & 1.2295 \\ 
 $10^5$ & 1.6254 & 1.2299 & 1.6254 & 1.2300 & 1.8330 & 1.5642 & 1.6254 & 1.2299 \\ 
   \hline
\end{tabular}
\end{center}
\end{table}

\begin{table}[ht]
\caption{Comparison of the maximum likelihood estimate of the Poisson point process model with $q=2$ and the binomial regression model
with the cauchit (inverse of Cauchy) link function.
The sample is (\ref{eq:sample}) and $n$ is fixed to $10$.
The normalizing sequence $(c_m,d_m)$ is (\ref{eq:Cauchy-c_m}).
}
\label{table:jikken2}
\begin{center}
\begin{tabular}{|c|cc|cc|}
  \hline
  & \multicolumn{2}{|c|}{Poisson process} & \multicolumn{2}{|c|}{cauchit}\\
 $m$ & $\hat{\alpha}$ & $\hat{\beta}$ & $\hat{\alpha}$ & $\hat{\beta}$ \\ 
  \hline
 $10^2$ & 0.8662 & 0.0667 & 0.8632 & 0.0656 \\ 
 $10^3$ & 0.8626 & 0.0673 & 0.8623 & 0.0677 \\ 
 $10^4$ & 0.8622 & 0.0680 & 0.8622 & 0.0679 \\ 
 $10^5$ & 0.8621 & 0.0680 & 0.8622 & 0.0679 \\ 
   \hline
\end{tabular}
\end{center}
\end{table}

\section{Estimation of the $q$-exponential family of intensity measures} \label{section:estimator}

We deal with estimation problem of
the $q$-exponential family of intensity measures (\ref{eq:q-exponential-intensity}).
The maximum likelihood estimator is likely to fail to exist for small sample size $n$.
We propose a penalized maximum likelihood estimator.

We put the following assumption for simplicity.

\begin{assumption} \label{assumption:support}
 The covariate distribution $F(dx)$ is known.
 The support of $F$, denoted by $S(F)$, is finite,
 and is not included in any hyperplane in $\mathbb{R}^p$.
 The observable data $\{x_i\}_{i=1}^n$ belongs to $S(F)$.
\end{assumption}

In practice, $F(dx)$ may be replaced with the empirical, or estimated, distribution
based on the covariate sample $\{X_i\}_{i=1}^m$ of the original regression problem.

The parameter space is
\begin{align}
 \Theta = \{(\alpha,\beta)\mid 1+(1-q)(\alpha+\beta^\T x)>0\ \mbox{for\ any}\ x\in S(F)\}.
  \label{eq:Theta}
\end{align}
The set $\Theta$ is convex and unbounded since
it is intersection of half spaces including
the set $\{(\alpha,0)\mid 1+(1-q)\alpha>0\}$.
Furthermore, $\Theta$ is open since $S(F)$ is compact.
In terms of convex analysis, $\Theta$ corresponds to the polar set of $S(F)$.
See \citet{Barvinok2002}.

We consider a penalized log-likelihood function
\begin{align}
  -\Lambda_q(\alpha,\beta)
  + \sum_{i=1}^n \log\exp_q(\alpha+\beta^\T x_i)
  + \kappa\int \log\exp_q(\alpha+\beta^\T x)F(dx),
  \label{eq:add-one}
\end{align}
where $\kappa$ is a non-negative regularization parameter.
If $\kappa=0$, (\ref{eq:add-one}) is the log-likelihood function; see (\ref{eq:limit-likelihood}).
The penalty term represents a pseudo-data of size $\kappa$ distributed according to $F$.
The function (\ref{eq:add-one}) is concave with respect to $(\alpha,\beta)$ if $0\leq q\leq 1$.
Indeed, we can directly confirm that $-\exp_q(z)$ is concave if $q\geq 0$,
and that $\log(\exp_q(z))$ is concave if $q\leq 1$.

\begin{definition} \label{def:add-one}
We call the maximizer of (\ref{eq:add-one})
 the additive-smoothing estimator.
\end{definition}

This estimator has a desirable property as shown in the following example,
even if $q=1$.

\begin{example} \label{ex:independent-Poisson}
Let $F$ be a two-point distribution on $\mathbb{R}$ defined by
\[
 F(x=0) = p_0 \quad \mbox{and} \quad F(x=1) = p_1,
\]
where $p_0,p_1>0$ and $p_0+p_1=1$.
Denote the intensity at $x=0$ and $x=1$ by
$\lambda_0=p_0\exp_q(\alpha)$ and $\lambda_1=p_1\exp_q(\alpha+\beta)$, respectively.
It is not difficult to show that $(\alpha,\beta)\in\Theta$
corresponds one-to-one with $(\lambda_0,\lambda_1)\in\mathbb{R}_+^2$,
where $\mathbb{R}_+$ is the set of positive numbers.
Hence the model is equivalent to the independent Poisson observable model
with intensity $(\lambda_0,\lambda_1)$, regardless of $q$.
 Then the penalized log-likelihood (\ref{eq:add-one}) becomes
 \[
  -\lambda_0-\lambda_1 + n_0\log\lambda_0 + n_1\log\lambda_1
 + \kappa\left(
 p_0 \log \frac{\lambda_0}{p_0} + p_1\log \frac{\lambda_1}{p_1}
 \right),
 \]
 where $n_j$ denotes the number of observations $x_i=j$, $j\in\{0,1\}$.
 The additive-smoothing estimator is $\hat\lambda_j=n_j+\kappa p_j$, $j\in\{0,1\}$.
 If $\kappa>0$, then $(\hat\lambda_0,\hat\lambda_1)\in\mathbb{R}_+^2$
 and the estimator $(\hat\alpha,\hat\beta)$ always exists.
 Furthermore, if $0<\kappa\leq 1$,
 this estimator is known to be admissible
 with respect to the Kullback-Leibler loss function;
 see \citet[Theorem 1]{GhoshYang1988}.
For the same reason, if $S(F)$ has only $p+1$ points in $\mathbb{R}^p$,
then the additive-smoothing estimator is admissible as long as $0<\kappa\leq 1$.
\end{example}

 Let $q=1$ and $F$ be any distribution satisfying Assumption~\ref{assumption:support}.
 Then, since the model (\ref{eq:limit-likelihood}) is an exponential family,
 the pair $(n,\bar{x}_n)$ is a sufficient statistic,
 where $\bar{x}_n=n^{-1}\sum_{i=1}^n x_i$ is the sample mean.
 Indeed, the additive-smoothing estimator should satisfy
\begin{align}
 \Lambda_1(\hat\alpha,\hat\beta) = n+\kappa
 \quad\mbox{and}\quad
 \frac{\int x e^{\hat\beta^\T x}F(dx)}{\int e^{\hat\beta^\T x}F(dx)} = 
 \frac{n\bar{x}_n + \kappa\int x F(dx)}{n+\kappa}.
 \label{eq:add-one-exp}
\end{align}
For the maximum likelihood estimator, meaning $\kappa=0$,
the second equation of (\ref{eq:add-one-exp}) is consistent with the result of \citet{Owen2007}.
From the theory of exponential families,
the solution to (\ref{eq:add-one-exp}) always exists if $\kappa>0$ since $\int xF(dx)$
belongs to the interior of the convex hull of $S(F)$; see \citet[Corollary 9.6]{Barndorff-Nielsen1978}.
On the other hand, the maximum likelihood estimator fails to exist if $\bar{x}_n$ is a boundary point.

For $q\neq 1$, we provide a similar result on existence.
First consider the following example.
The pair $(n,\bar{x}_n)$ is not a sufficient statistic any more.

\begin{example} \label{ex:fail}
 Let $q=0$ and $F$ be a three-point distribution on $\mathbb{R}$ defined by
 $F(x=j)=1/3$ for $j\in\{0,1,2\}$.
 Denote the number of observations $x_i=j$ by $n_j$.
 We use $\theta=1+\alpha$ and $\phi=1+\alpha+2\beta$ as a new parameter.
 Then the parameter space is $\theta>0$ and $\phi>0$.
 The penalized log-likelihood is
 \begin{align}
  -\frac{\theta+\phi}{2} 
  + n_0^* \log \theta
  + n_1^* \log \frac{\theta+\phi}{2}
  + n_2^* \log \phi,
  \label{eq:fail-1}
 \end{align}
 where $n_j^*=n_j+\kappa/3$.
 The maximizer $(\hat\theta,\hat\phi)$ of (\ref{eq:fail-1}) is
 \[
 \hat\theta = \frac{2n_0^*(n_0^*+n_1^*+n_2^*)}{n_0^*+n_2^*}
 \quad \mbox{and} \quad
 \hat\phi = \frac{2n_2^*(n_0^*+n_1^*+n_2^*)}{n_0^*+n_2^*}.
 \]
 This always belongs to the parameter space if $\kappa>0$.
 On the other hand, the maximum likelihood estimator
 fails to exist if $n_0=0$ or $n_2=0$.
\end{example}

In general, the following theorem holds.
The proof is given in Appendix.

\begin{theorem} \label{theo:add-one}
 Let $q$ be any real number and $\kappa>0$.
 If Assumption~\ref{assumption:support} is satisfied,
 then the additive-smoothing estimator exists almost surely.
 It is unique if $0\leq q\leq 1$.
 \end{theorem}

 \section{Discussion} \label{section:discussion}
 
\subsection{Multinomial regression}

We studied so far the binomial regression.
There are variants of multinomial regression models.
The multinomial $t$-logistic regression proposed by \cite{Ding2011}
can be proved to have a limit under imbalanced asymptotics
in the same manner as Theorem~\ref{theo:t-logistic}.
The author was not aware of more general results.
The problem is postponed as a future work.

\subsection{Convergence of estimator}

We did not study convergence properties of estimators such as the maximum likelihood estimator.
Instead we considered the additive-smoothing estimator for the $q$-exponential family of intensity measures in Section~\ref{section:estimator}.

\cite{Owen2007} showed that the maximum likelihood estimator of the logistic regression
converges to that of the exponential family under imbalanced asymptotics.
Then a natural conjecture is
that the maximum likelihood estimator of the binomial regression model, which is the maximizer of
\[
 \sum_{i=1}^m 
 \left[
 Y_i\log G(a+b^\T X_i) + (1-Y_i)\log\{1-G(a+b^\T X_i)\}
 \right],
\]
converges to that of the $q$-exponential family.
Note that estimation of $(a,b)$
is equivalent to that of $(\alpha,\beta)$ via the formula (\ref{eq:a_m-b_m}).
It will be also meaningful
to study convergence of statistical experiments;
see \citet{vanderVaart1998} for the terminology.

An estimator corresponding to the additive-smoothing estimator of Definition~\ref{def:add-one}
is the maximizer of
\[
\sum_{i=1}^m
\left[ Y_i\log G(a+b^\T X_i) + (1-Y_i)\log\{1-G(a+b^\T X_i)\}
\right]
+ \frac{\kappa}{m} \sum_{i=1}^m \log\{mG(a+b^\T X_i)\}
\]
since the additional term converges to $\kappa\int \log \exp_q(\alpha+\beta^\T x)F(dx)$
after normalization (\ref{eq:a_m-b_m}).
The estimator is expected to converge as well.

\subsection{Misspecified case}

We studied asymptotic properties of the binomial regression model under
an assumption that the model (\ref{eq:model}) is true.
On the other hand, \cite{Owen2007}
put a different assumption,
in that the true conditional distribution of the covariate $X_i$
given $Y_i=j$, $j\in\{0,1\}$, is fixed to some distribution $F_j$.
In this assumption, our setting is asymptotically described as
$F_0(dx)=F(dx)$ and $F_1(dx)=\{\exp_q(\alpha+\beta^\T x)/\Lambda_q(\alpha,\beta)\}F(dx)$
by (\ref{eq:limit-likelihood}).
In other words, if the true distributions $F_j$ do not satisfy this relation,
the model is misspecified.

It is important to consider robustness of estimators
under the misspecified assumption.
The problem is not so serious if the support of $F_1$ is included in that of $F$,
since then $F_1$ is absolutely continuous with respect to
the estimated intensity measure $\exp(\hat\alpha+\hat\beta^\T x)F(dx)$,
whenever $(\hat\alpha,\hat\beta)$ belongs to the parameter space (\ref{eq:Theta}).
Otherwise, however, 
$F_1$ is not absolutely continuous. 
In other words, the estimated intensity measure
does not allow that the future data $x_{n+1}$ falls into a region.
In particular, if the support of $F_1$ is not assumed a priori,
there is risk of such a contradiction.

One may consider to take a distribution $F$ with the full support $\mathbb{R}^p$
in order to contain the support of $F_1$.
However, if $q\neq 1$, we cannot assume such a distribution $F$
since the parameter space (\ref{eq:Theta}) becomes $\{(\alpha,0)\mid 1+(1-q)\alpha>0\}$.

A solution to this problem will be to use a parametric family of $F$ together with a Bayesian prior distribution.
For example, let $F(dx)=F(dx\mid \theta)$ be the uniform distribution on the hypercube $[-\theta,\theta]^p$,
and assume a prior density on $\theta>0$.
As long as the true $F_1(dx)$ has compact support,
we have a chance to detect it
since there is a sufficiently large $\theta$ such that the support of $F_1$
is included in that of $F(\cdot\mid \theta)$.

\subsection{Bayesian prediction}

In the preceding subsection,
we considered the Bayesian approach for
treating misspecified case.
Even if the model is correctly specified,
the approach will be fruitful.

In Section~\ref{section:estimator},
we considered the additive-smoothing estimator of $(\alpha,\beta)$.
This is considered as
a maximum-a-posteriori estimator if the prior density
\[
\pi(\alpha,\beta) = \exp\left(\kappa\int \log\exp_q(\alpha+\beta^\T x)F(dx)\right)
\]
is adopted.
Then additive-smoothing Bayesian prediction can be also defined by the same prior.

In Example~\ref{ex:independent-Poisson}, we noted that,
for special cases of $F$ and $\kappa$, the additive-smoothing estimator
becomes an admissible estimator with respect to the Kullback-Leibler divergence,
shown by \cite{GhoshYang1988}.
For prediction problem, a class of admissible predictive densities
is investigated by \cite{Komaki2004}.
Together with the additive-smoothing estimator,
decision-theoretic properties of the additive-smoothing prediction
are of interest.

\section*{Acknowledgement}

The author thanks to Saki Saito for helpful discussions in the exploratory stage.

\appendix
\section{Appendix}

\subsection{Proof of Lemma~\ref{lem:Owen-mar-cond}}

 Denote the induced probability distribution of $t=\alpha+\beta^\T X_i$
 by $F^*(dt)$.
 Let $A^*$ be $A^*=\{\alpha+\beta^\T x\mid x\in A\}$.
 Then $A^*$ is compact since $A$ is.
 We have
 \begin{align*}
  P_{m,\alpha,\beta}(Y_i=1,X_i\in A)
  &= \int_A G(a_m(\alpha)+b_m(\beta)^\T x) F(dx)
  \\
  &= \int_A G(c_m+d_m(\alpha+\beta^\T x)) F(dx)
  \\
  &= \int_{A^*} G(c_m+d_m t) F^*(dt).
 \end{align*}
 To prove (\ref{eq:lem-marginal}),
 it is enough to show that
 \[
 \int_{A^*} G(c_m + d_m t) F^*(dt) = \frac{1}{m}\int_{A^*} \exp_q(t) F^*(dt) + o(m^{-1}).
 \]
 By Assumption~\ref{assumption:GEV},
 we know $mG(c_m + d_m t)=\exp_q(t) + o(1)$ for each $t\in A^*$.
 Hence it is enough to show that
 $mG(c_m + d_m t)$ converges to $\exp_q(t)$ uniformly in $t\in A^*$.
 However, since $mG(c_m+d_m t)$ is monotone in $t$
 and $\exp_q(t)$ is continuous in $t\in A^*$,
 uniform convergence follows from the general argument;
 see e.g.\ \citet[Lemma 2.10.1]{Galambos1987}).

\subsection*{Proof of Theorem~\ref{theo:t-logistic}}

 For each real number $q$,
 denote the set of distributions that satisfy Assumption~\ref{assumption:GEV} by ${\cal D}_q$.

 For  $t=1$, elementary calculation shows that
 $G_1(z)=e^z/(1+e^z)$.
 This is the logistic distribution and belongs to ${\cal D}_1$.

 For $t=0$, we have $G_0(z)=(1+z)/2$, $-1<z<1$.
 This is the uniform distribution on $[-1,1]$
 and belongs ${\cal D}_0$.
 
 Let $t>1$.
 It suffices to show that
 \[
 G_t(z) = [(1-t)z]^{1/(1-t)} + o((-z)^{1/(1-t)}),
 \quad z\to -\infty.
 \]
 Indeed, by the condition (\ref{eq:t-logistic-cond}),
 if $z\to -\infty$, then $\gamma_t(z)\to 0$. Thus
 \begin{align*}
  \exp_t(z - \gamma_t(z))
  &= \exp_t(z + o(z))
  \\
  &= \left[ 1 + (1-t)(z+o(z)) \right]^{1/(1-t)}
  \\
  &= [(1-t) z]^{1/(1-t)} + o((-z)^{1/(1-t)})
 \end{align*}
 Hence $G_t$ belongs to ${\cal D}_t$.
 
 For $t<1$, we first show that the support of $G_t$ has the infimum $z_*=-1/(1-t)$
 and that $\gamma_t(z)$ tends to $0$ as $z\to z_*+0$.
 Note that the $t$-exponential function $\exp_t(z)$ is continuous in $z\in\mathbb{R}$,
 strictly increasing over $z>z_*$,
 and remains $0$ over $z\leq z_*$.
 Since $\exp_t(z)>1$ for any $z>0$,
 it must be $\gamma_t(z)\geq 0$ for any $z\in\mathbb{R}$ by (\ref{eq:t-logistic-cond}).
 Then $\exp_t(z-\gamma_t(z))>0$ only if $z>z_*$.
 Conversely, if $z>z_*$,
 it must be $\exp_t(z-\gamma_t(z))>0$.
Indeed, if $\exp_t(z-\gamma_t(z))=0$,
then $\gamma_t(z)=0$ by (\ref{eq:t-logistic-cond}),
 but this contradicts $z>z_*$.
 To prove $\gamma_t(z)\to 0$ as $z\to z_*+0$, due to (\ref{eq:t-logistic-cond}),
 it is sufficient to show that $\exp_t(z-\gamma_t(z))\to 0$ as $z\to z_*+0$.
 This is shown as
 \[
  0\leq \exp_t(z-\gamma_t(z))
  \leq \exp_t(z) \to 0,
  \quad z\to z_*+0.
 \]
 
 Let $0<t<1$ and $z_*=-1/(1-t)$.
 It suffices to show that
 \begin{align}
 G_t(z) &= [(1-t)(z-z_*)]^{1/(1-t)} + o((z-z_*)^{1/(1-t)}),
 \quad z\to z_*+0.
 \label{eq:t-logistic-0}
 \end{align}
 By the definition of $z_*$, we have
 \begin{align}
 \exp_t(z-\gamma_t(z))
 &= [1+(1-t)(z-\gamma_t(z))]^{1/(1-t)}
 \nonumber\\
 &= [(1-t)(z-z_*-\gamma_t(z))]^{1/(1-t)}.
 \label{eq:t-logistic-1}
 \end{align}
 On the other hand,
 since $\gamma_t(z)\to 0$ as $z\to z_*+0$, we obtain
 \begin{align}
 \exp_t(-\gamma_t(z)) = 1 - \gamma_t(z) + o(\gamma_t(z)).
  \label{eq:t-logistic-2}
 \end{align}
 By substituting the two equations to (\ref{eq:t-logistic-cond}),
 we obtain $\gamma_t(z)=O((z-z_*)^{1/(1-t)})=o(z-z_*)$.
Then (\ref{eq:t-logistic-1}) implies (\ref{eq:t-logistic-0}).
 Hence $G_t$ belongs to ${\cal D}_t$.

 Finally, let $t<0$ and $z_*=-1/(1-t)$.
 We show that $G_t$ belongs to ${\cal D}_0$, not ${\cal D}_t$.
 It suffices to show that
 \begin{align}
  G_t(z) = (z-z_*) + o(z-z_*),\quad z\to z_*+0.
   \label{eq:t-logistic-3}
 \end{align}
 For the same reason as the case $0<t<1$,
 we have the two equations (\ref{eq:t-logistic-1})
 and (\ref{eq:t-logistic-2}).
 By substituting them to (\ref{eq:t-logistic-cond}),
 we obtain
 \[
  \gamma_t(z) = (z-z_*) - \frac{(z-z_*)^{1-t}}{1-t} + o((z-z_*)^{1-t}).
 \]
 Then (\ref{eq:t-logistic-1}) implies (\ref{eq:t-logistic-3}).
 Hence $G_t$ belongs to ${\cal D}_0$.

\subsection*{Proof of Theorem~\ref{theo:add-one}}

Uniqueness follows from concavity of (\ref{eq:add-one}) for $0\leq q\leq 1$.
We prove the existence result.
Since the case $q=1$ is proved in (\ref{eq:add-one-exp}),
we assume $q\neq 1$.

In the following, we prove the theorem only for the case that $n=0$,
that is, no data is observed.
The case $n\geq 1$ is similarly proved
if one notes that $\{x_i\}_{i=1}^n$
is contained in the convex hull of the support of $F$.

Let $F$ be a discrete distribution with support $\{\xi_j\}_{j=1}^J\subset\mathbb{R}^p$
and put $p_j=F(x=\xi_j)>0$, $j\in\{1,\ldots,J\}$.
By assumption, $\{\xi_j\}_{j=1}^J$ is not included in any hyperplane of $\mathbb{R}^p$.
The parameter space (\ref{eq:Theta}) is written as
\[
\Theta=\{(\alpha,\beta)\mid 1+(1-q)(\alpha+\beta^\T \xi_j)>0,\ j\in\{1,\ldots,J\}\}.
\]
Note that $\Theta$ is an open convex set and the origin $(\alpha,\beta)=(0,0)$ always belongs to $\Theta$.
The penalized log-likelihood is, since $n=0$,
\begin{align}
L(\alpha,\beta) =
 \sum_{j=1}^J p_j\left\{ - \exp_q(\alpha+\beta^\T \xi_j) + \log\exp_q(\alpha+\beta^\T \xi_j)\right\}.
 \label{eq:theo-add-one-1}
\end{align}
By continuity of $L(\alpha,\beta)$ over $\Theta$,
it is sufficient to show that $L(\alpha,\beta)\to -\infty$
if $(\alpha,\beta)$ tends to a boundary point of $\Theta$ or $(\alpha,\beta)$ diverges.
Note that if $(\alpha_0,\beta_0)$ is a boundary point of $\Theta$,
then $(t\alpha_0,t\beta_0)$ belongs to $\Theta$ for any $0\leq t<1$ since the origin does.

We prove the claim for $q<1$ first, and then $q>1$.

Let $q<1$. Fix any boundary point $(\alpha_0,\beta_0)$ of $\Theta$.
Then there is at least one $\xi_j$ such that $\exp_q(\alpha_0+\beta_0^\T \xi_j)=0$.
For such $\xi_j$'s, $\exp_q(t(\alpha_0+\beta_0^\T \xi_j))\to +0$ as $t\to 1-0$.
For the other $\xi_j$'s, $\exp_q(t(\alpha_0+\beta_0^\T \xi_j))$ is bounded as $t\to 1-0$.
Then, by (\ref{eq:theo-add-one-1}),
the function $L(t\alpha_0,t\beta_0)$ tends to $-\infty$ as $t\to 1-0$.

Let $q<1$ and fix any $(\alpha_1,\beta_1)\in\Theta\setminus\{(0,0)\}$
such that $(t\alpha_1,t\beta_1)\in\Theta$ for any $t>0$.
Then it is necessary that $\alpha_1+\beta_1^\T \xi_j\geq 0$ for all $j$.
Since $\{\xi_j\}$ is not contained in a hyperplane,
there is at least one $\xi_j$ such that $\alpha_1+\beta_1^\T \xi_j>0$.
For such $\xi_j$'s, we have $\exp_q(t\alpha_1+t\beta_1^\T \xi_j)\to \infty$ as $t\to\infty$.
For the other $\xi_j$'s, $\exp_q(t\alpha_1+t\beta_1^\T \xi_j)=\exp_q(0)=1$.
Therefore, by (\ref{eq:theo-add-one-1}),
the function $L(t\alpha_1,t\beta_1)$ tends to $-\infty$ as $t\to\infty$,
and the case $q<1$ was completed.

Let $q>1$.
Fix any boundary point $(\alpha_0,\beta_0)$ of $\Theta$.
Then there is at least one $\xi_j$ such that $\exp_q(\alpha_0+\beta_0^\T \xi_j)=\infty$.
For such $\xi_j$'s, $\exp_q(t(\alpha_0+\beta_0^\T \xi_j))\to\infty$ as $t\to 1-0$.
For the other $\xi_j$'s, $\exp_q(t(\alpha_0+\beta_0^\T \xi_j))$ is bounded as $t\to 1-0$.
Then, by (\ref{eq:theo-add-one-1}),
the function $L(t\alpha_0,t\beta_0)$ tends to $-\infty$ as $t\to 1-0$.

Finally, let $q>1$ and fix any $(\alpha_1,\beta_1)\in\Theta\setminus\{(0,0)\}$
such that $(t\alpha_1,t\beta_1)\in\Theta$ for any $t>0$.
Then it is necessary that $\alpha_1+\beta_1^\T \xi_j\leq 0$.
Since $\{\xi_j\}$ is not contained in a hyperplane,
there is at least one $\xi_j$ such that $\alpha_1+\beta_1^\T \xi_j<0$.
For such $\xi_j$'s, $\exp_q(t\alpha_1+t\beta_1^\T \xi_j)\to +0$ as $t\to\infty$.
For the other $\xi_j$'s, $\exp_q(t\alpha_1+t\beta_1^\T \xi_j)=\exp_q(0)=1$.
Therefore, by (\ref{eq:theo-add-one-1}),
the function $L(t\alpha_1,t\beta_1)$ tends to $-\infty$ as $t\to\infty$,
and the case $q>1$ was completed.


\end{document}